\documentclass[11pt, a4paper,reqno]{amsart}

\usepackage{color} \definecolor{bleu_sombre}{rgb}{0,0,0.6}  \definecolor{rouge_sombre}{rgb}{0.8,0,0}\definecolor{vert_sombre}{rgb}{0,0.6,0}
\usepackage[plainpages=false,colorlinks,linkcolor=bleu_sombre,
citecolor=rouge_sombre,urlcolor=vert_sombre,breaklinks]{hyperref}

\usepackage[english]{babel}

\usepackage{amsmath,amssymb,amsthm,graphicx,amsfonts,url,color,enumerate,dsfont,stmaryrd, mathrsfs, csquotes}

\theoremstyle{plain}
\newtheorem{theorem}{{Theorem}}[section] 
\newtheorem*{theorem*}{{Theorem}}
\newtheorem{proposition}[theorem]{Proposition}
\newtheorem*{proposition*}{Proposition}

\newtheorem*{corollary*}{Corollary}
\newtheorem{lemma}[theorem]{Lemma}
\newtheorem{assumption}[theorem]{Assumption}
\newtheorem*{lemma*}{Lemma}
\theoremstyle{definition}

\newtheorem*{definition*}{Definition}
\theoremstyle{remark}
\newtheorem{remark}[theorem]{Remark}

\makeatletter

\@addtoreset{equation}{section}  
\makeatother

\renewcommand{\leq}{\leqslant}	\renewcommand{\geq}{\geqslant}
\renewcommand{\bar}[1]{\overline{#1}}
\newcommand{\abs}[1]{\left\vert #1\right\vert}        
\newcommand{\nr}[1]{\left\Vert #1\right\Vert}  
\newcommand{\innp}[2]{\left< #1 , #2 \right>}         

\newcommand{\Ii}[2] {\llbracket #1,#2 \rrbracket}	
\newcommand{\R}{\mathbb{R}}	\newcommand{\C}{\mathbb{C}}
\newcommand{\N}{\mathbb{N}}	
\newcommand{\Z}{\mathbb{Z}}

\newcommand{\st}{\,:\,}	
			
\newcommand{\mBV} {m_{\mathbf{B},V}}
\newcommand{\eps}{\varepsilon}
\renewcommand{\O}{\Omega}
\renewcommand{\a}{\alpha}
\renewcommand{\b}{\beta}

\renewcommand{\Re}{\mathrm{Re}\,}
\renewcommand{\Im}{\mathrm{Im}\,}

\newcommand{\rot}{\mathop{\mathrm{rot}}\nolimits}
\newcommand{\grad}{\mathop{\mathrm{grad}}\nolimits}
\newcommand{\der}{\mathrm{d}}
\newcommand{\Dom}{\mathsf{Dom}}
\newcommand{\Ker}{\mathsf{Ker}}
\newcommand{\Ran}{\mathsf{Ran}}
\newcommand{\supp}{\mathop{\mathrm{supp}}\nolimits}
\newcommand{\ind}{\mathop{\mathrm{ind}}\nolimits}
\newcommand{\codim}{\mathop{\mathrm{codim}}\nolimits}
\usepackage[normalem]{ulem}                      
\definecolor{DarkGreen}{rgb}{0,0.5,0.1}

\newcommand\soutD{\bgroup\markoverwith
{\textcolor{DarkGreen}{\rule[.5ex]{2pt}{1pt}}}\ULon}
\newcommand{\Hm}[1]{\leavevmode{\marginpar{\tiny%
$\hbox to 0mm{\hspace*{-0.5mm}$\leftarrow$\hss}%
\vcenter{\vrule depth 0.1mm height 0.1mm width \the\marginparwidth}%
\hbox to 0mm{\hss$\rightarrow$\hspace*{-0.5mm}}$\\\relax\raggedright
#1}}}

\begin{document}

\title[Non-accretive Schr\"odinger operators]
{Non-accretive Schr\"odinger operators and exponential decay of their eigenfunctions}

\author{D. Krej{\v{c}}i{\v{r}}{\'{\i}}k}
\address[D. Krej{\v{c}}i{\v{r}}{\'{\i}}k]{
Department of Theoretical Physics, Nuclear Physics Institute ASCR, 
25068 \v{R}e\v{z}, Czech Republic}
\email{krejcirik@ujf.cas.cz}

\author{N. Raymond}
\address[N. Raymond]{IRMAR, Universit\'e de Rennes 1, Campus de Beaulieu, F-35042 Rennes cedex, France}
\email{nicolas.raymond@univ-rennes1.fr}

\author{J. Royer}
\address[J. Royer]{Institut de math\'ematiques de Toulouse, Universit\'e Toulouse 3, 118 route de Narbonne, F-31062 Toulouse cedex 9, France}
\email{julien.royer@math.univ-toulouse.fr}

\author{P. Siegl}
\address[P. Siegl]{
Mathematical Institute, 
University of Bern,
Alpeneggstrasse 22,
3012 Bern, Switzerland
\& On leave from Nuclear Physics Institute ASCR, 25068 \v Re\v z, Czech Republic}
\email{petr.siegl@math.unibe.ch}

\subjclass[2010]{}


\thanks{
This work was partially supported by the IUF grant of S. V\~u Ng\d{o}c. The research of P.S. is supported by the \emph{Swiss National Foundation}, SNF Ambizione grant No. PZ00P2\_154786. D.K.\ was  supported by the project RVO61389005 and the GACR grant No.\ 14-06818S.
}

\date{}

\begin{abstract}
We consider non-self-adjoint electromagnetic Schr\"odinger operators 
on arbitrary open sets with complex scalar potentials 
whose real part is not necessarily bounded from below.
Under a suitable sufficient condition on the electromagnetic potential,
we introduce a Dirichlet realisation as a closed densely defined operator 
with non-empty resolvent set
and show that the eigenfunctions corresponding to discrete eigenvalues
satisfy an Agmon-type exponential decay. 
\end{abstract}

\maketitle

\section{Introduction}
%

\subsection{Context and motivation}
We consider the electromagnetic Schr\"odinger operator
\begin{equation}\label{operator.intro}
  (-i\nabla+\mathbf{A})^2+V 
  \qquad \mbox{in} \qquad L^2(\Omega)
  \,,
\end{equation}
subject to Dirichlet boundary conditions on~$\partial\Omega$,
where~$\Omega$ is an arbitrary open subset of~$\R^d$.
The functions $V:\Omega\to\C$ and $\mathbf{A}:\Omega\to\R^d$
are the scalar (electric) and vector (magnetic) potentials, respectively.

If $d=3$ and~$V$ is real-valued, 
the self-adjoint Dirichlet realisation of~\eqref{operator.intro}
is the Hamiltonian of a quantum particle 
constrained to a nanostructure~$\Omega$
and subjected to an external electromagnetic field 
$(-\grad V,-\rot \mathbf{A})$. The literature on the subject is enormous 
and we restrict ourselves to
referring to the recent book~\cite{Raymond} 
with an extensive bibliography.

Although complex-valued potentials~$V$ have appeared in quantum theory
from its early years, too, notably in the context of effective Hamiltonians
for open systems (see, e.g., \cite{Exner-open}) and resonances 
(see \cite{Abramov-Aslanyan-Davies_2001} for a more recent study),
the corresponding spectral theory is much less developed.
The interest in non-self-adjoint Schr\"odinger operators
have been renewed at the turn of the millenium
with the advent of the so-called quasi-Hermitian quantum mechanics 
(see \cite{KS-book} for a mathematically oriented review).
There are also motivations in other areas of physics,
for instance, superconductivity 
(see~\cite{AH2014} for a mathematical treatement)
and optics with a number of recent experiments 
(see, e.g., \cite{Regensburger_2012}).
Finally, Schr\"odinger operators with potentials 
having a complex coupling constant (in fact spectral parameter) appear naturally
in the study of the damped wave equation (see, e.g., \cite{sjostrand00,boucletr14}).

\subsection{About the main results}
Our main result is the Agmon-type exponential decay of eigenfunctions corresponding to discrete eigenvalues of \eqref{operator.intro}, cf.~Theorem~\ref{theo.main}, which can be viewed as a non self-adjoint version of the Agmon-Persson estimates, see \cite{Persson60, Agmon85}. We emphasise that the decay is \emph{not} an effect of the positive part of $\Re V$ since it may be absent, or even worse, $\Re V$ is allowed to be negative and unbounded at infinity. 

\subsubsection{A sufficient condition to define the operator} 
The first problem that we tackle in our analysis is finding of a Dirichlet realisation of \eqref{operator.intro} with \emph{non-empty resolvent set}. This is not a trivial task as we do not restrict the signs of $\Re V$ and $\Im V$ 
and so the standard sectorial form techniques of \cite[Sec.~VI.2.1]{Kato} are not available. 

A simple example one should have in mind is
\begin{equation}\label{paradigma}
  -\frac {\der^2}{\der x^2} - x^2 + i x^3
  \qquad \mbox{in} \qquad L^2(\R)
  \,,
\end{equation}
for which the numerical range covers the whole complex plane. Due to the latter, even the Kato's theorem for accretive Schr\"odinger operators, based on Kato's distributional inequality \cite[Sec.~VII.2]{Edmunds-Evans}, is not applicable immediately\footnote{Note that, in special self-adjoint settings, however, interesting alternative approaches can be found in the literature. For instance, in~\cite{Kostrykin_2013}, representation theorems for indefinite quadratic forms are established and can be used to define certain self-adjoint
operators possibly unbounded from below.}. Here we can even go beyond operators like \eqref{paradigma} for which the suitable Dirichlet realisation can be actually found by available methods in \cite{AlmogHe15,Boegli-2015}. We allow much wilder behaviour of $V$ in terms of the possible growth at infinity and oscillations. In more detail, we essentially require that (cf.~Assumption~\ref{a1} and Proposition~\ref{prop.a0})
\begin{align}
|\nabla V(x)| + |\nabla \mathbf B (x)| &= o \left ((|V(x)|+|\mathbf B(x)|)^\frac32+1 \right), \label{nabla.V.32}
\\
(\Re V(x))_- & = o \Big(|V(x)|+|\mathbf B(x)|+1 \Big),
\label{ReV.neg}
\end{align}
as $|x| \to \infty$, where $(\Re V)_-$ is the negative part of $\Re V$ 
and $\mathbf B := \der \mathbf{A}$ is the magnetic matrix. 

The condition \eqref{ReV.neg} puts restrictions on the size of $(\Re V)_-$ which in fact represents a ``small'' perturbation of an m-accretive operator \eqref{operator.intro} with $V$ replaced by $(\Re V)_+ + i \, \Im V$. Notice however, that $(\Re V)_-$ can be compensated not only by $\Im V$, but also by the magnetic field. 
In a different context (absence of eigenvalues), 
a certain analogy between the magnetic field and $\Im V$ was observed in~\cite{FKV}.

\subsubsection{About the power $\frac{3}{2}$}
The power $\frac{3}{2}$ in the condition \eqref{nabla.V.32} 
is an improvement comparing to \cite{AlmogHe15,Boegli-2015}
where the power~$1$ is assumed; in these references (where \eqref{paradigma} fits already),
a big-$\mathcal O$ instead of the little-$o$ is used. In the present paper,
we can therefore treat examples like
\begin{equation}\label{paradigma.2}
  -\frac {\der^2}{\der x^2} - e^{x^2} + i e^{x^4}
  \qquad \mbox{in} \qquad L^2(\R)
  \,.
\end{equation}
Moreover, we show in Theorem~\ref{theo.sep} that the operator domain of 
the found realisation of~\eqref{operator.intro} possesses a very convenient separation property, namely
\begin{equation}
\Dom( (-i\nabla+\mathbf{A})^2+V ) = \Dom( (-i\nabla+\mathbf{A})^2) \cap \Dom(V)\,.
\end{equation}
The power $\frac{3}{2}$ in~\eqref{nabla.V.32} 
is not a coincidence as it is known to be optimal (with little-$o$ replaced by a sufficiently small constant in \eqref{nabla.V.32}) with respect to the separation property in the self-adjoint  case \cite{Evans-1978-8,Everitt-1978-79,Brown-1999-124} (see also \cite{Duf83}, \cite{HM88} in the magnetic case).

\subsubsection{Weighted coercivity}
Our approach for proving all the results of this paper is based on the generalised Lax-Milgram-type theorem of Almog and Helffer \cite{AlmogHe15} involving a new idea of \emph{weighted coercivity}, which can be viewed as a generalisation of the $\mathbb T$-coercivity (see for instance \cite[Def.~2.1]{Bonnet-BenDhia-2010-234}). While from the point of view of abstract Lax-Milgram or representation theorems, an optimal ``if and only if'' condition for m-accretivity was found in the recent work \cite[Thm.~4.2]{Elst-2015-269}, the weighted coercivity of Theorem~\ref{theo-minoration} makes such abstract results  directly applicable for~\eqref{operator.intro}. 
Moreover, the present paper reveals a connection between weighted coercivity and exponential decay of eigenfunctions stated 
in Theorem~\ref{theo.main}.

\subsection{Examples of applications}
Besides the independent interest of our results, 
we indicate below two connections to other recent works, 
both when~$|V|$ is confining so that the resolvent of~\eqref{operator.intro} is compact
(see Proposition~\ref{prop.comp-res}).
The first one concerns the completeness of eigensystem of \eqref{operator.intro}, the second one the rates of eigenvalue convergence of domain truncations.

\subsubsection{Eigensystem completeness} The crucial ingredient in a natural proof of the eigensystem completeness is the fundamental result of operator theory
(see, e.g., \cite[Cor.~XI.9.31]{DS2})
combining the $p$-Schatten class property of the resolvent and a control of the resolvent norm on a sufficient number of rays in~$\C$; for operators like \eqref{operator.intro}, this approach was followed in \cite{SK,AlmogHe15}. We indicate how the completeness results can be extended to operators satisfying weaker conditions \eqref{nabla.V.32} only. Our domain separation and the graph norm estimate (cf.~Theorem~\ref{theo.sep}), 
the second resolvent identity and the ideal property of Schatten classes show that the resolvent of~\eqref{operator.intro} is in the $p$-Schatten class ($0 < p <\infty$) if and only if the resolvent of the self-adjoint~\eqref{operator.intro} with~$V$ replaced by~$|V|$ is in the $p$-Schatten class;
to obtain the value of~$p$ depending on~$V$ and~$\mathbf A$, criteria of the type \cite[Thm.~1.3]{AlmogHe15} can be applied. 
To have the control of the resolvent norm on rays in~$\C$, 
we can use the standard bound ($1$ over the distance to the numerical range) 
if~\eqref{operator.intro} is at least accretive and, in the non-accretive case, the perturbation result \cite[Thm.~IV.3.17]{Kato} with viewing $(\Re V)_-$ as a relatively bounded perturbation of an m-accretive operator \eqref{operator.intro} with~$V$ replaced by $(\Re V)_+ + i \; \Im V$ (see \cite[Prop.~2.4 (iv)]{Boegli-2015} for details on such an approach). 

\subsubsection{Domain truncation} It was proved in \cite{Boegli-2015} that eigenvalues of \eqref{operator.intro} on $\R^d$ with $\mathbf A = 0$ and $V$ satisfying (stronger) conditions of the type \eqref{nabla.V.32}--\eqref{ReV.neg}, see \cite[Asm.~II]{Boegli-2015}, can be approximated without pollution by the eigenvalues of \eqref{operator.intro} truncated to a sequence of expanding domains, e.g.~balls, 
and subject to Dirichlet boundary conditions. The rate of convergence for a given eigenvalue or \eqref{operator.intro} on $\R^d$ was estimated by the decay rate of the corresponding eigenfunctions (and generalised eigenfunctions in the case of Jordan blocks) 
at infinity (see~\cite[Thm.~5.2]{Boegli-2015}). Our Agmon-type estimate, cf.~Theorem~\ref{theo.main} and Remark~\ref{rem.Agm}, shows that this convergence is exponential which vastly generalises known facts for complex polynomial potentials 
(see, e.g., \cite{Sibuya-1975,Cappiello-2010-111}).

\subsection{Organisation of the paper}
In Section~\ref{Sec.Results}, we summarise our main results.
The definition of~\eqref{operator.intro} as a closed
densely defined operator together with a convenient characterisation
of the operator domain is performed in Section~\ref{sec.1}.
The spectral properties are established in Section~\ref{sec.3}.
At the end of the paper, we attach Appendix~\ref{sec.A}
with elements of spectral theory related to the present study.

\section{Main results}\label{Sec.Results}
%

\subsection{Assumptions}
Let $\O$ be a non-empty open (possibly unbounded) subset of $\R^d$, $d \geq 1$.
Another standing assumption of this paper is that the electromagnetic 
potentials satisfy  
$$
  (V,\mathbf{A})
  \in\mathcal{C}^1(\bar\Omega;\C)\times \mathcal{C}^2(\bar\Omega;\R^d)
  \,.
$$
This smoothness hypothesis is technically convenient,
but it is definitely far from being optimal for 
the applicability of our techniques
and the validity of the obtained results.  
We write
$
  V=V_{1}+iV_{2}
$
where~$V_{1}$ and~$V_{2}$ are real-valued.
Associated with the vector potential $\mathbf{A}$, 
we consider the magnetic (skew-symmetric) matrix
\begin{equation}\label{tensor}
  \mathbf{B} = (B_{jk})_{j,k=1}^d 
  \,, \qquad 
  B_{jk} := \partial_j A_k - \partial_k A_j = i [P_j,P_k]
  \,,
\end{equation}
where $P_\ell := -i\partial_\ell + A_\ell$.

As in \cite{AlmogHe15}, let us introduce functions
\begin{equation}\label{weights}
  \Phi := \frac {V_{2}}{\mBV} 
  \qquad \mbox{and} \qquad
  \Psi := \frac {\mathbf{B}}{\mBV}
  \,,
\end{equation}
where
\[
  \mBV := \sqrt{1 + \abs{\mathbf{B}}^2 + \abs V^2 }
  \,.
\]
Here $|V(x)|$ denotes the usual norm of a complex number, 
while we use
\[
  \abs{\mathbf{B}(x)} := \sqrt{\sum_{j,k=1}^d B_{jk}(x)^2}
  \,, \qquad
  \abs{\nabla \mathbf{B}(x)}
  := \sqrt{\sum_{j,k = 1}^d \abs{\nabla B_{jk} (x)}^2}
  \,,
\]
where $\abs{\nabla B_{jk}(x)}$ is now the usual Euclidean norm
of a vector in~$\R^d$. 
Finally, given a real-valued function~$a$,
we adopt the standard notation
$a_{\pm}:=\max(\pm a,0)$.

With these notations, the main hypothesis of this paper reads:
\begin{assumption}\label{a1}
There exist constants $\gamma_{1}>0$ and $\gamma_{2}\in\R$ such that
\begin{equation}\label{eq.a0}
 \frac {V_{2}^2 +\frac{1}{12d} \abs{\mathbf{B}}^2}{\mBV} + V_{1} -  9 \left( \abs {\nabla \Phi}^2 + \abs{\nabla \Psi}^2 \right) \geq \gamma_{1}|V|-\gamma_{2}\,.
\end{equation}
\end{assumption}

On the left-hand side in~\eqref{eq.a0}, 
the first term is non-negative, 
the last bracket gives a non-positive contribution
and~$V_1$ has no sign \emph{a priori}. 
If~$V_1$ is bounded from below, 
then we only have to control the last term to obtain the required inequality. 
The point is that~$V_2$ or~$\mathbf{B}$ can be used to control 
the non-positive contribution of~$V_{1}$.

Assumption~\ref{a1} is easily checked to hold for~\eqref{paradigma}.
A sufficient condition for the validity of Assumption~\ref{a1}
is contained in the following proposition.

\begin{proposition}\label{prop.a0}
We assume
\begin{align} 
  \abs{\nabla V(x)} + \abs{\nabla \mathbf{B}(x)} 
  &=
  \mathop{o} \big( \mBV^{\frac{3}{2}}(x) \big)
  \,,
  \label{hyp-nabla}
  \\
  \left(V_{1}\right)_{-}(x)
  &= \mathop{o}\big(\mBV(x)\big)
  \,,
  \label{eq.Im-dom}
\end{align}
as $|x|\to+\infty$. 
Then Assumption~\ref{a1} is satisfied.
\end{proposition}

\subsection{Definition of the operator}
First we introduce the usual magnetic Sobolev space
\[
  H^1_{\mathbf{A}}(\O)
  :=\{u\in L^2(\O)\st (-i\nabla+\mathbf{A})u\in L^2(\O)\}
  \,,
\]
equipped with the norm
\[
  \|u\|_{H^1_{\mathbf{A}}(\O)}
  := \sqrt{\|u\|^2+\|(-i\nabla+\mathbf{A})u\|^2}
  \,.
\]
Here~$\|\cdot\|$ denotes the norm of $L^2(\O)$
and the associated inner product will be denoted by $\innp{\cdot}{\cdot}$.
We also introduce the subspace $H^1_{\mathbf{A},0}(\O)$ 
defined as the closure of $\mathcal{C}^\infty_{0}(\O)$ 
for the norm $\|\cdot\|_{H^1_{\mathbf{A}}(\O)}$.
Then we can introduce our variational space as 
\[
  \mathscr{V} := 
  \left\{ u \in H_{\mathbf{A},0}^1(\O) \st  
  \abs{V}^{\frac 12} u \in L^2(\O) \right\}
  \,,
\]
equipped with the norm
\[
  \nr{u}_{\mathscr{V}} 
  := \sqrt{ \nr{u}_{H^1_{\mathbf{A}}(\O)}^2 
  + \int_\O \abs V \abs u^2 \, \der x }
  \,,
\]
with respect to which~$\mathscr{V}$ is complete.

We introduce a sesquilinear form
\[
  Q(u,v) := 
  \innp{(-i\nabla + \mathbf{A} ) u}{{(-i\nabla + \mathbf{A}) v}}
  + \int_\O V u \bar v \, \der x
  \,, \qquad
  \Dom(Q) := \mathscr{V}
  \,.
\]
For $u,v \in \mathcal{C}^\infty_{0}(\O)$,
a dense subspace of~$\mathscr{V}$,
we have 
\[
  Q(u,v) = \innp{(-i\nabla+\mathbf{A})^2 u+Vu}{v}\,,
\]
so~$Q$ is the form naturally associated with~\eqref{operator.intro}. 
If~$V$ were such that~$Q$ was sectorial, then~$Q$ would be closed  
and it would give rise to an m-sectorial operator by 
Kato's representation theorem \cite[Thm.~VI.2.1]{Kato}.
In our general setting (where the numerical range of~$Q$
is allowed to be the whole complex plane), however, 
there is no general representation theorem 
and even the notion of closedness for forms is not standard.
Anyway, we are still allowed to introduce \emph{an} operator~$\mathscr{L}$
by the Riesz theorem
\begin{equation}\label{operator}
  \forall u \in \mathsf{Dom}(\mathscr{L}),
  \quad \forall v \in \mathscr{V}, \qquad 
  Q(u,v) =: \innp{\mathscr{L} u}{v}
  \,,
\end{equation}
where
\begin{equation}\label{domain}
  \mathsf{Dom}(\mathscr{L}) := 
  \left\{ v \in \mathscr{V} \st  u \mapsto Q (u,v) 
  \text{ is continuous on $\mathscr{V}$ for the norm of $L^2(\O)$} \right\}
  \,.
\end{equation}

The following theorem shows that such a defined operator~$\mathscr{L}$
shares all the nice properties of operators introduced by the standard
representation theorem.
The proof is based on the new abstract representation theorem of Almog and Helffer (see \cite[Thm.~2.2]{AlmogHe15}, reproduced below as Theorem~\ref{th-lax-milgram}).

\begin{theorem}\label{prop.def-op}
Suppose Assumption~\ref{a1}.
The following properties hold:
\begin{enumerate}
\item[\emph{(i)}] 
$\mathsf{Dom}(\mathscr{L})$ is dense in $\mathsf{H}$,
\item[\emph{(ii)}] 
$\mathscr{L}$ is closed,
\item[\emph{(iii)}]  
the resolvent set of $\mathscr{L}$ is not empty.
\end{enumerate}
\end{theorem}

Furthermore, we have the following description of the domain of~$\mathscr{L}$.
\begin{theorem}\label{theo.sep}
Let~\eqref{hyp-nabla} and~\eqref{eq.Im-dom} hold.
Then we have
\begin{equation*}
  \mathsf{Dom}(\mathscr L) 
  = \left\{ 
  u \in \mathscr V \, : \, 
  (-i \nabla + \mathbf{A})^2 u \in L^2(\O)
  \ \land \
  V u \in L^2(\O)
  \right\}
  \,.
\end{equation*}
Moreover, for all $\delta>0$, 
there exists $C_\delta>0$ such that, 
for all $u \in \mathsf{Dom}(\mathscr L)$, 
\begin{equation}\label{L.graph.nom}
  \|\mathscr L u \|^2 
  \geq (1-\delta) \left(
  \|(-i \nabla + \mathbf{A})^2 u \|^2  + \|V u \|^2
  \right) 
  - C_\delta \|u\|^2.
\end{equation}
\end{theorem}

\subsection{Spectral properties}
The reader may wish to consult Appendix~\ref{sec.A}, 
where we recall basic definitions related to the spectrum
and Fredholm properties.

First of all, we give a sufficient condition for~$\mathscr{L}$
to have a purely discrete spectrum.
\begin{proposition}\label{prop.comp-res}
Suppose Assumption~\ref{a1}.
If 
\begin{equation}\label{confine}
  \displaystyle{\lim_{|x|\to+\infty}|V(x)|=+\infty}
  \,,
\end{equation}
then~$\mathscr{L}$ is an operator with compact resolvent.
\end{proposition}

In general, we give an estimate on the location of the essential spectrum.
To this purpose, let us introduce the quantity 
(which is either a finite non-negative number or infinity)
\[
  V_{\infty} := \liminf_{\abs{x} \to+\infty} \abs{V(x)}
  \,,
\]
and the following family of subsets of the complex plane:
\[
  \rho_{c}
  :=\left\{\mu\in\C\st -c-\Re\mu-|\Im\mu|>0\right\}
  \,,
\]
where~$c$ is any real number.
\begin{theorem}\label{prop.discrete}
Suppose Assumption~\ref{a1}.
We have
\begin{equation}\label{dis1}
\rho_{\gamma_{2}}\subset\rho(\mathcal{\mathscr{L}})\,.
\end{equation}
Moreover, assuming that~$V_{\infty}$ is positive, 
we have
\begin{equation}\label{dis2}
  \rho_{\gamma_{2}}
  \subset\rho_{\gamma_{2}-\gamma_{1}\check V_{\infty}}
  \subset\mathsf{Fred}_{0}(\mathscr{L})
\end{equation}
for all $\check V_{\infty}\in\left(0,V_{\infty}\right)$.
The spectrum of $\mathscr{L}$ 
contained in $\rho_{\gamma_{2}-\gamma_{1}\check V_{\infty}}$, 
if it exists, is formed by isolated eigenvalues 
with finite algebraic multiplicity.
\end{theorem}
\begin{remark}
When $V_{\infty}=+\infty$, 
we recover from Theorem~\ref{prop.discrete}
the result of Proposition~\ref{prop.comp-res}.
\end{remark}

Finally, we state our main result. 
It shows in particular that the discrete spectrum 
in the region $\rho_{\gamma_{2}-\gamma_{1}\check V_{\infty}}$
is associated with exponentially decaying eigenfunctions 
and that this decay may be estimated in terms of an Agmon-type distance. 
\begin{theorem}\label{theo.main}
Suppose Assumption~\ref{a1}.
Let us assume that 
$
  \mathsf{sp}(\mathscr{L}) \cap
  \rho_{\gamma_{2}-\gamma_{1}\check V_{\infty}}
  \neq\emptyset
$ 
and consider~$\lambda$ in this set. Let us define the metric
\[
  g(x) := \left(
  \gamma_{1}|V(x)|- \Re(\lambda) - \abs{\Im(\lambda)}-\gamma_{2}
  \right)_{+}  \, \der x^2
  \,,
\]
and the corresponding \emph{Agmon distance} 
(to any fixed point of~$\Omega$) 
$\der_{\mathsf{Ag}}(x)$ that satisfies 
\begin{equation}\label{eq.Agmon}
  |\nabla \der_{\mathsf{Ag}}|^2
  =\left(\gamma_{1}|V|- \Re(\lambda) - \abs{\Im(\lambda)}-\gamma_{2}\right)_{+}
  \,.
\end{equation}
Pick up any $\eps\in(0,1)$. 
If~$\psi$ is an eigenfunction associated with~$\lambda$, 
we have
\begin{equation}\label{eq.est-Agmon}
  e^{\frac{1-\eps}{3}\der_{\mathsf{Ag}}}\,\psi\in L^2(\Omega)
  \,.
\end{equation}
The same conclusion holds for all~$\psi$ 
in the algebraic eigenspace associated with~$\lambda$.
\end{theorem}
\begin{remark}\label{rem.Agm}
If there exist $R>0$ and $\gamma > 0$ such that, 
\[
  \forall |x| \geq R
  \,,\qquad 
  \gamma_{1}|V|- \Re(\lambda) - \abs{\Im(\lambda)}-\gamma_{2}\geq \gamma 
  \,,
\]
then there exists $M\geq 0$ such that, 
in this region, $\der_{\mathsf{Ag}}(x)\geq \gamma\,|x|-M$.
\end{remark}
%

\section{Weighted coercivity and representation theorems}\label{sec.1}
%
The main objective of this section is to prove 
Theorems~\ref{prop.def-op} and~\ref{theo.sep}.

\subsection{Two abstract representation theorems}
We first recall the following generalised representation theorems from \cite{AlmogHe15}.

\begin{theorem}[{\cite[Thm.~2.1]{AlmogHe15}}] \label{th-lax-milgram0}
Let $\mathcal{V}$ be a Hilbert space.
Let $Q$ be a continuous sesquilinear form on $\mathcal{V} \times \mathcal{V}$. Assume that there exist $\Phi_1,\Phi_2 \in \mathcal{L}(\mathcal{V})$ and $\a > 0$ such that for all $u \in \mathcal{V}$ we have 
\begin{align*}
\abs{Q(u,u)} + \abs{Q(\Phi_1(u),u)} &\geq \a \nr{u}_\mathcal{V}^2 \,,
\\
\abs{Q(u,u)} + \abs{Q(u,\Phi_2(u))} &\geq \a \nr{u}_\mathcal{V}^2 \,.
\end{align*}
The operator $\mathscr{A}$ defined by 
\[
\forall u,v \in \mathcal{V}, \quad Q(u,v) = \innp{\mathscr{A} u}{v}_\mathcal{V}
\]
is a continuous isomorphism of $\mathcal{V}$ onto $\mathcal{V}$
with bounded inverse.
\end{theorem}

\begin{theorem}[{\cite[Thm.~2.2]{AlmogHe15}}] \label{th-lax-milgram}
In addition to the hypotheses of Theorem~\ref{th-lax-milgram0},
assume that~$\mathsf{H}$ is a Hilbert space such that 
$\mathcal{V}$ is continuously embedded and dense in~$\mathsf{H}$ 
and that~$\Phi_1$ and~$\Phi_2$ extend to bounded operators on~$\mathsf{H}$.
Then the operator~$\mathscr{L}$ defined by 
\[
\forall u \in \mathsf{Dom}(\mathscr{L}), \quad
\forall v \in \mathcal{V}, 
\qquad Q(u,v) =: \innp{\mathscr{L} u}{v}_\mathsf{H}
\]
where
\[
  \mathsf{Dom}(\mathscr{L}) := 
  \left\{ u \in \mathcal{V} \st \text{the map } v \mapsto Q (u,v) 
  \text{ is continuous on $\mathcal{V}$ for the norm of $\mathsf{H}$} \right\}
  \,,
\]
satisfies the following properties:
\begin{enumerate}
\item[\emph{(i)}] 
$\mathscr{L}$ is bijective from $\mathsf{Dom}(\mathscr{L})$ onto $\mathsf{H}$,
\item[\emph{(ii)}] 
$\mathsf{Dom}(\mathscr{L})$ is dense in $\mathcal{V}$ and in $\mathsf{H}$,
\item[\emph{(iii)}] 
$\mathscr{L}$ is closed.
\end{enumerate}
\end{theorem}

\subsection{Weighted coercivity estimates}\label{sec.WCE}
For any complex number~$\mu$,
consider the shifted form
$Q_\mu (u,v):= Q(u,v) - \mu \innp{u}{v}$.
The aim of this subsection is to prove the following estimate
and deduce Theorem~\ref{prop.def-op} from it (with help of Theorem~\ref{th-lax-milgram}).
\begin{theorem}[Weighted coercivity] \label{theo-minoration}
For every $\mu\in\C$, 
$W \in W^{1,\infty}(\Omega;\R)$ 
and all $u \in \mathcal{C}^\infty_{0}(\Omega)$, we have 
\begin{multline*}
  \Re \big[ Q_{\mu} (u, e^{2W} u) \big] 
  + \Im \big[ Q_{\mu} (u, \Phi e^{2W} u) \big] 
  \geq \frac 12 \nr{(-i\nabla +  \mathbf{A}) e^W u}^2
  \\
   + \int_\O \abs{e^W u}^2  
  \left[\frac {V_{2}^2 +\frac{1}{12d} \abs{\mathbf{B}}^2}{\mBV} + V_{1}-\Re\mu-|\Im\mu| -  9 \left( \abs {\nabla \Phi}^2 + \abs{\nabla \Psi}^2 + \abs{\nabla W}^2 \right) 
  \right] \der x
  \,.
\end{multline*}
\end{theorem}
 
In order to prove Theorem \ref{theo-minoration},
we need two lemmata.

\begin{lemma} \label{lem-BmBV}
For every $u \in\mathcal{C}^\infty_{0}(\Omega)$, 
we have 
\[
  \int_{\O} \frac {\abs{\mathbf{B}}^2}{\mBV} \abs u^2 \, \der x 
  \leq 3d \nr{(-i\nabla + \mathbf{A}) u}^2 + \nr{(\nabla \Psi) u}^2
  \,.
\]
\end{lemma}
\begin{proof}
Let $u \in\mathcal{C}^\infty_{0}(\Omega)$ 
and $j,k \in \Ii 1 d := [1,d] \cap \Z$.
Using~\eqref{tensor} and~\eqref{weights}, we have 
\begin{eqnarray*}
  \lefteqn{\int_{\O} \frac {B_{jk}^2}{\mBV} \abs u^2 \, \der x  
  = \big\langle i[P_j,P_k] u,\Psi_{jk}u \big\rangle 
  = \big\langle i P_k u,P_j\Psi_{jk}u \big\rangle 
  - \big\langle i P_j u,P_k\Psi_{jk}u \big\rangle} 
  \\
  &&= \big\langle i P_k u,\Psi_{jk} P_j u \big\rangle 
  - \big\langle i P_j u,\Psi_{jk} P_k u \big\rangle 
  - \big\langle P_k u,(\partial_j\Psi_{jk}) u \big\rangle 
  + \big\langle P_j u,(\partial_k \Psi_{jk}) u \big\rangle 
  \\
  && \leq \frac 32 \nr{P_j u}^2  + \frac 32 \nr{P_k u}^2 
  + \frac 12 \nr{(\partial_j \Psi_{jk}) u}^2 
  + \frac 12 \nr{(\partial_k \Psi_{jk}) u}^2
  \,.
\end{eqnarray*}
We conclude by summing over $j,k \in \Ii 1 d$.
\end{proof}

The second lemma follows elementarily by an integration by parts.
\begin{lemma}\label{lem.loc}
For every $u \in \mathcal{C}^\infty_{0}(\Omega)$ 
and $\chi \in W^{1,\infty}(\O;\R)$, 
we have
\[
\Re \innp{(-i\nabla+\mathbf{A}) u}{(-i\nabla+\mathbf{A}) \chi^2 u} 
= \nr{(-i\nabla+\mathbf{A}) \chi u}^2 - \nr{(\nabla \chi) u}^2\,.
\]
\end{lemma}

Now we are in a position to prove Theorem~\ref{theo-minoration}.
\begin{proof}[Proof of Theorem~\ref{theo-minoration}]
Let us consider $u \in\mathcal{C}^\infty_{0}(\Omega)$ 
and $W \in W^{1,\infty}(\O;\R)$. 
Choosing $\chi := e^W$ in Lemma~\ref{lem.loc}, 
we get the identity
\begin{equation} \label{minor-Re-a}
  \Re \big[ Q (u, e^{2W} u) \big] 
  =  \int_{\Omega} V_{1} \abs{e^W u}^2 \, \der x 
  + \nr{(- i \nabla + \mathbf{A}) e^W u}^2 - \nr{(\nabla W) e^W u}^2
  \,.
\end{equation}
Moreover, we have
\begin{align*}
  \Im \big[ Q (u, \Phi e^{2W} u) \big]
  & = \Im \innp{(-i\nabla +  \mathbf{A}) u}{(-i\nabla +  \mathbf{A})(\Phi e^{2W} u)} 
  + \int_{\Omega} \frac {V_{2}^2}{\mBV} \abs{e^W u}^2 \, \der x
  \,.
\end{align*}
The first term of the right-hand side equals
\begin{eqnarray*}
\lefteqn{
\Im \innp{(-i\nabla +  \mathbf{A}) u}{-i (\nabla \Phi + 2 \Phi \nabla W) e^{2W} u)}
}
\\ 
&& = 
\Im \innp{e^W (-i\nabla +  \mathbf{A}) u}{-i (\nabla \Phi + 2 \Phi \nabla W) e^{W} u)}
\\ 
&& = 
\Im \innp{ (-i\nabla +  \mathbf{A})e^W u}{-i (\nabla \Phi + 2 \Phi \nabla W) e^{W} u)}
\,.
\end{eqnarray*}
Consequently, for all $\a \in (0,1)$, we have
\begin{multline*}
\abs{\Im \innp{(-i\nabla +  \mathbf{A}) u}{(-i\nabla +  \mathbf{A})(\Phi e^{2W} u)}}\\
\leq \a \nr{(-i\nabla +  \mathbf{A}) e^W u}^2 + \frac 1 {4\a} \nr{\big( \nabla \Phi + 2 (\nabla W) \Phi \big) e^W u}^2
\end{multline*}
and therefore
\begin{multline} \label{minor-Im-a}
\Im \big[ Q (u, \Phi e^{2W} u) \big] \\
\geq \int_{\Omega} \frac {V_{2}^2}{\mBV} \abs{e^W u}^2 \, \der x 
- \a \nr{(-i\nabla +  \mathbf{A}) e^W u}^2 
- \frac 1 {4\a} \nr{\big( \nabla \Phi + 2 (\nabla W) \Phi \big) e^W u}^2
\,.
\end{multline}
Summing up~\eqref{minor-Re-a} and~\eqref{minor-Im-a},
we deduce
\begin{multline*}
\Re \big[ Q (u, e^{2W} u) \big] 
+ \Im \big[ Q (u, \Phi e^{2W} u) \big] 
\geq (1-\a) \nr{(-i\nabla +  \mathbf{A}) e^W u}^2\\
   + \int_\O \abs{e^W u}^2  \left(\frac {V_{2}^2}{\mBV} + V_{1} 
- \abs{\nabla W}^2 - \frac 1 {2\a} \abs{\nabla \Phi}^2 
- \frac 2 \a \abs{\nabla W}^2 \right) \, \der x
\,.
\end{multline*}
It remains to add the term involving $\abs{\mathbf{B}}^2$. 
By Lemma \ref{lem-BmBV}, we have 
\[
\nr{(-i\nabla +  \mathbf{A}) e^W u}^2  
\geq \frac 1 {3d} \left(
\int_{\O} \frac {\abs{\mathbf{B}}^2}{\mBV} \abs {e^W u}^2 \, \der x 
- \nr{(\nabla \Psi) e^{W} u}^2
\right)
\,.
\]
Thus, for all $\b \in [0,1-\a]$, we get
\begin{multline*}
\Re \big[ Q (u, e^{2W} u) \big] + \Im \big[ Q (u, \Phi e^{2W} u) \big]
\geq (1-\a - \b) \nr{(-i\nabla +  \mathbf{A}) e^W u}^2
\\
   + \int_\O \abs{e^W u}^2  \left(\frac {V_{2}^2 
+ \frac \b {3d} \abs {\mathbf{B}}^2}{\mBV} + V_{1} - \frac {2+\a} \a\abs{\nabla W}^2 
- \frac 1 {2\a} \abs{\nabla \Phi}^2 
- \frac \b {3d} \abs{\nabla \Psi}^2 \right) \, \der x
\,.
\end{multline*}
The proof is concluded by taking $\a=\b=\frac{1}{4}$ 
and adding the contribution related to the shift by~$\mu$.
\end{proof}

With Theorems~\ref{th-lax-milgram} and~\ref{theo-minoration} 
we easily deduce Theorem~\ref{prop.def-op}.
\begin{proof}[Proof of Theorem~\ref{prop.def-op}]
Under Assumption~\ref{a1}, the inequality of Theorem~\ref{theo-minoration}
extends to all $u \in \mathscr{V}$.
Applied with $W = 0$, Theorem~\ref{theo-minoration} then gives, 
for all $u \in \mathscr{V}$,
\begin{multline}\label{eq.coer-Qmu}
\abs{ Q_{\mu} (u, u) } + \abs{ Q_{\mu} (u, \Phi u) }
\geq \frac 12 \nr{(-i\nabla +  \mathbf{A}) u}^2 
\\
+ \int_\O \abs{u}^2  \left(\frac {V_{2}^2 + \frac 1 {12d} \abs B^2}{\mBV}
+ V_{1}-\Re\mu-|\Im\mu|  - 9 (\abs{\nabla \Phi}^2 +  \abs{\nabla \Psi}^2) \right) 
\, \der x\,.
\end{multline}
Using Assumption~\ref{a1}, it implies
\begin{equation}\label{eq.coer-Qmu'}
\begin{aligned}
  \abs{ Q_{\mu} (u, u) } + \abs{ Q_{\mu} (u, \Phi u) }
 & \geq \frac 12 \nr{(-i\nabla +  \mathbf{A})u}^2
\\
 & \quad +\int_{\Omega} \left( \gamma_{1}|V|-\Re\mu-|\Im\mu|-\gamma_{2} \right) |u|^2
  \, \der x \,.
\end{aligned}
\end{equation}
Taking $\mu\in\R$ such that $\mu<-\gamma_{2}$, 
the inequality establishes the coercivity of~$Q_{\mu}$ on~$\mathscr{V}$, 
so it is enough to apply Theorem~\ref{th-lax-milgram} to~$Q_{\mu}$.
\end{proof}

\subsection{Description of the operator domain}
At this moment, we only know that the operator domain of~$\mathscr L$
is given by~\eqref{domain}.
This subsection is devoted to a proof of Theorem~\ref{theo.sep},
which gives a more explicit characterisation of $\Dom(\mathscr L)$.

Let us first state a density result.
\begin{lemma}\label{lem.density}
The set
\begin{equation}\label{core.def}
  \mathcal{D} := \left\{ u \in \mathsf{Dom}(\mathscr L) \, : \
  \supp u \ \text{ is compact in } \overline \Omega \right\} 
\end{equation}
is a core of~$\mathscr L$.
\end{lemma}
\begin{proof}
From the definition of $\mathsf{Dom}(\mathscr L)$ given by~\eqref{domain}, 
we get that
\begin{equation}\label{subset}
  \mathsf{Dom}(\mathscr L) 
  \subset 
  \{
  u \in \mathscr V \, : \, (-i  \nabla + \mathbf{A})^2 u + V u \in L^2(\Omega)
  \}
  \,.
\end{equation}
Take $u \in \mathsf{Dom}(\mathscr L)$ and notice that 
$V u \in  L^2_{\rm loc}(\overline \Omega)$ 
from our regularity assumption about~$V$, 
thus $(-i  \nabla + \mathbf{A})^2 u \in L^2_{\rm loc}(\overline \Omega)$ as well. 
We define a suitable cut-off, 
see \cite[proof of Thm.~8.2.1]{Davies-1995}. 
Consider a non-negative function $\varphi \in \mathcal{C}^\infty_{0}(\R^d)$ 
such that $\varphi(x) =1$ if $|x|<1$ and $\varphi(x)=0$ if $|x|>2$ and, 
for $u \in \mathsf{Dom}(\mathscr L)$, define, for all $x\in\Omega$ and $n\in\N$,
\begin{equation}
  u_n(x) := u(x) \varphi_n(x)
  \,, \qquad 
  \varphi_n(x):=\varphi \left(\frac xn \right)
  \,.
\end{equation}
Since
\begin{equation}
  (-i  \nabla + \mathbf{A})^2 u_n 
  = \varphi_n (-i  \nabla + \mathbf{A})^2 u 
  - 2 i  \nabla \varphi_n \cdot (-i  \nabla + \mathbf{A}) u 
  - (\Delta \varphi_n) u
  \,,
\end{equation}
we have from the derived regularity of~$u$ 
and the compactness of $\supp \varphi_n$ 
that $\{u_n\}_{n \in \N} \subset \mathcal{D}$. 
Moreover, by the dominated convergence theorem, 
$\|u_n - u\| \to 0$ as $n \to \infty$ and 
\begin{multline*}
  \|[(-i  \nabla + \mathbf{A})^2 + V] u  
  - [(-i  \nabla + \mathbf{A})^2 + V] u_n \|
  \\ 
  \leq 
  \|(1-\varphi_n)[(-i  \nabla + \mathbf{A})^2 + V]u\| 
  +2 \| \nabla \varphi_n \cdot (-i  \nabla + \mathbf{A}) u\|  
  + \|(\Delta \varphi_n) u \| \xrightarrow[n \to \infty]{} 0
  \,,
\end{multline*}
since 
$
  \|\nabla \varphi_n\|_{L^\infty(\R^d)} 
  = n^{-1} \|\nabla \varphi\|_{L^\infty(\R^d)}
$, 
$
  \|\Delta \varphi_n\|_{L^\infty(\R^d)} 
  = n^{-2} \|\Delta \varphi\|_{L^\infty(\R^d)}
$ 
and $u \in \mathscr V$.
\end{proof}

By integrating by parts, we get the following lemma.
\begin{lemma}\label{lem.nabla.A}
For all $u\in\mathcal{D}$ and $\delta>0$, we have
\begin{equation*}
  2 \, \| (-i  \nabla + \mathbf{A}) u\|^2  
  \leq \delta\| (-i  \nabla + \mathbf{A})^2 u\|^2 +\delta^{-1} \|u\|^2
  \,.
\end{equation*}
\end{lemma}
\begin{proof}
For every $u\in\mathcal{D}$, we have
$$
  \| (-i  \nabla + \mathbf{A}) u\|^2 
  = \innp{(-i  \nabla + \mathbf{A}) u}{(-i  \nabla + \mathbf{A}) u} 
  = \innp{(-i  \nabla + \mathbf{A})^2 u}{u} 
$$ 
where the second equality employs an integration by parts
using our regularity assumptions about~$V$ and~$\mathbf{A}$,
namely $V u \in  L^2(\Omega)$ with~\eqref{subset}. 
The proof is concluded by applying the Cauchy-Schwarz and Young inequalities.
\end{proof}
In the following Lemma~\ref{lem.B^2-0} and Proposition~\ref{lem.B^2},
we establish estimates on $|\mathbf{B}|u$; the proofs are adaptations of \cite[Lem.~3.4]{AlmogHe15}.
\begin{lemma}\label{lem.B^2-0}
Suppose~\eqref{hyp-nabla}.
There exists $C>0$ such that, for all $u\in\mathcal{D}$,
\begin{equation}\label{Bu.est}
  \big\||\mathbf{B}| u\big\|^2 
  \leq C \left( \| m_{\mathbf{B},V}^\frac 12  
  (-i  \nabla + \mathbf{A}) u\|^2 +\|V u\|^2 
  + \| (-i  \nabla + \mathbf{A})^2 u\|^2+ \|u\|^2 \right) 
  \,.
\end{equation}
\end{lemma}
\begin{proof}
Let $u \in \mathcal D$. 
Then $B_{jk} u \in \mathscr V$ and similarly as in Lemma~\ref{lem-BmBV}, 
we have
\begin{equation}\label{B.est.1}
  \|B_{jk} u\|^2 = \Im \langle [P_j,P_k] u, B_{jk} u \rangle
  \leq 
  |\langle P_k u, P_j B_{jk} u  \rangle| + |\langle P_j u, P_k B_{jk} u  \rangle|
\end{equation}
for every $j,k \in \Ii 1 d$.
Further, using the assumption \eqref{hyp-nabla}, we get that, 
for all $\eps_{1}>0$, there exist $C_{\eps_{1}}, \tilde C_{\eps_{1}}>0$ 
such that
\begin{equation}\label{B.est.2}
\begin{aligned}
|\langle P_k u, P_j B_{jk} u  \rangle| &\leq
|\langle B_{jk} P_k u, P_j u \rangle| 
+ |\langle P_k u, u \partial_j B_{jk} \rangle|
\\
&\leq 
\||B_{jk}|^\frac12 P_k u\| \, \||B_{jk}|^\frac12 P_j u\| 
+ \eps_1 \|m_{\mathbf{B},V}^\frac 12 P_k u\| \, \|m_{\mathbf{B},V} u\|
\\
& \quad + C_{\eps_1} \| P_k u\| \, \|u\|
\\
& \leq 
\||B_{jk}|^\frac12 P_k u\| \, \||B_{jk}|^\frac12 P_j u\| + 
\\
& \quad 
+ 
\eps_1 \left(\|m_{\mathbf{B},V}^\frac 12 P_k u\|^2 +  \|m_{\mathbf{B},V} u\|^2 + \| P_k u\|^2   \right) + \tilde C_{\eps_1} \|u\|^2\,.
\end{aligned}
\end{equation}
Summing up over $j$ and $k$, 
we get from~\eqref{B.est.1} and~\eqref{B.est.2} 
that there exists $C_{1}>0$ such that, for all $\eps_1 \in(0,1)$, 
there exists $\hat C_{\eps_{1}}>0$ such that 
\begin{equation*}
\||\mathbf{B}| u\|^2 
\leq C_1 \left( \| m_{\mathbf{B},V}^\frac 12  
(-i  \nabla + \mathbf{A}) u\|^2 + \eps_1 \big(\|m_{\mathbf{B},V}u\|^2 
+ \| (-i  \nabla + \mathbf{A})u\|^2\big) \right)  + \hat C_{\eps_1} \|u\|^2
\,.
\end{equation*}
We now use Lemma~\ref{lem.nabla.A} to get the desired estimate.
\end{proof}
\begin{proposition}\label{lem.B^2}
Suppose~\eqref{hyp-nabla}.
There exists  $C>0$, such that, for all $u\in\mathcal{D}$, we have
\begin{equation}\label{B.12.lem}
\| |\mathbf{B}| u \|^2 + \| m_{\mathbf{B},V}^\frac 12  (-i  \nabla + \mathbf{A}) u\|^2
\leq 
C \left(\|(-i  \nabla + \mathbf{A})^2 u \|^2 + \|Vu\|^2 
+\|u\|^2 \right).
\end{equation}
\end{proposition}
\begin{proof}
Let us first show that, for all $\eps>0$, 
there exists $C_{\eps}>0$ such that, for all $u \in \mathcal D$,
\begin{equation}\label{m.A.int}
\| m_{\mathbf{B},V}^\frac 12  (-i  \nabla + \mathbf{A}) u\|^2
\leq 
C_{\eps} (\|(-i  \nabla + \mathbf{A})^2 u \|^2 + \|u\|^2) + \eps\|m_{\mathbf{B},V}u\|^2
\,.
\end{equation}
We write
\[
  \| m_{\mathbf{B},V}^\frac 12  
  (-i  \nabla + \mathbf{A}) u\|^2
  =\langle m_{\mathbf{B},V}(-i\nabla+\mathbf{A})u,(-i\nabla+\mathbf{A})u\rangle
  \,,
\]
so that, by an integration by parts,
\begin{equation}\label{eq.a}
\| m_{\mathbf{B},V}^\frac 12  
(-i  \nabla + \mathbf{A}) u\|^2
=\langle (-i\nabla m_{\mathbf{B},V})(-i\nabla+\mathbf{A})u,u\rangle
+\langle m_{\mathbf{B},V}(-i\nabla+\mathbf{A})^2u,u\rangle
\,.
\end{equation}
We have, for all $\eps_1\in(0,1)$,
\begin{equation}\label{eq.b}
| \langle m_{\mathbf{B},V}(-i\nabla+\mathbf{A})^2u,u\rangle|
\leq \frac{\eps_1}{2}\|m_{\mathbf{B},V}u\|^2
+\frac{1}{2\eps_1}\|(-i\nabla+\mathbf{A})^2u\|^2\,.
\end{equation}
Moreover, by using  \eqref{hyp-nabla} and Lemma \ref{lem.nabla.A}, 
for all $\eps_1\in(0,1)$, there exists $C_{\eps_{1}}>0$ such that
\begin{multline}\label{eq.c}
|\langle (-i\nabla m_{\mathbf{B},V})(-i\nabla+\mathbf{A})u,u\rangle|
\leq
\frac{\eps_1}{2}
\left(\|m_{\mathbf{B},V}^\frac 12  
(-i  \nabla + \mathbf{A}) u \|^2+\|m_{\mathbf{B},V}u\|^2
\right)
\\
+C_{\eps_1}\left(\|u\|^2+\|(-i\nabla+\mathbf{A})^2u\|^2\right)\,.
\end{multline}
Using~\eqref{eq.a}, \eqref{eq.b} and~\eqref{eq.c}, we deduce~\eqref{m.A.int}.
Having established~\eqref{m.A.int},
it remains to combine it with Lemma~\ref{lem.B^2-0} 
and choose~$\eps$ sufficiently small.
\end{proof}

Now we are in a position to establish Theorem~\ref{theo.sep}.
\begin{proof}[Proof of Theorem \ref{theo.sep}]
For all $u \in \mathcal{D}$, we have
\begin{equation}\label{mixed.est0}
\begin{aligned}
\|\mathscr L u\|^2  
& = \|(-i  \nabla + \mathbf{A})^2 u\|^2 
+  \|V u \|^2 + 2 \Re \langle (-i  \nabla + \mathbf{A})^2 u, V u \rangle
\\ 
& = \|(-i  \nabla + \mathbf{A})^2 u\|^2 +  \|V u \|^2 
+ 2 \Re \langle (-i  \nabla + \mathbf{A}) u, 
(-i  \nabla + \mathbf{A}) (V u) \rangle
\\  & 
\geq  
\|(-i  \nabla + \mathbf{A})^2 u\|^2 
+  \|V u \|^2 + 2\int_{\Omega}V_{1} |(-i  \nabla + \mathbf{A}) u|^2\, \der x
\\
& \quad
- 2 \langle |(-i  \nabla + \mathbf{A}) u|, |\nabla V| |u| \rangle
\,.
\end{aligned}
\end{equation}
Note that the second step is justified since $V u \in \mathscr V$. 
We proceed by estimating the last term of~\eqref{mixed.est0}. 
Let $\eps\in(0,1)$. There exist $C_{\eps}, \tilde C_{\eps}>0$ such that
\begin{equation}\label{mixed.est}
\begin{aligned}
2 \langle |(-i  \nabla + \mathbf{A}) u|, |\nabla V| |u| \rangle 
& \leq 2 \eps  \langle |(-i  \nabla + \mathbf{A}) u|, 
m_{\mathbf{B},V}^\frac 32 |u| \rangle 
+ 2 C_{\eps} \langle |(-i  \nabla + \mathbf{A}) u|, |u| \rangle
\\
& \leq 2\eps (\| m_{\mathbf{B},V}^\frac 12  
(-i  \nabla + \mathbf{A}) u\|^2 + \| m_{\mathbf{B},V} u \|^2)
+  \tilde C_{\eps} \|u\|^2\,.
\end{aligned}
\end{equation}
From~\eqref{mixed.est0}, \eqref{mixed.est}, \eqref{eq.Im-dom} 
and Lemma~\ref{lem.nabla.A}, we deduce that, for some $\hat C_{\eps}>0$,
\begin{equation}\label{eq.Lu}
\begin{aligned}
\|\mathscr L u\|^2 & \geq (1-2\eps) \left(\|(-i  \nabla + \mathbf{A})^2 u\|^2 
+  \|V u \|^2 \right) - 2\eps \||\mathbf{B}| u\|^2     
\\ & \quad - 3 \eps \|  m_{\mathbf{B},V}^\frac 12  
(-i  \nabla + \mathbf{A}) u\|^2  -\hat C_{\eps}\|u\|^2
\,.
\end{aligned}
\end{equation}
Finally, using Proposition~\ref{lem.B^2}, we get
\begin{equation}\label{eq.Lu.2}
\|\mathscr L u\|^2  \geq (1-2\eps - 3 C \eps ) 
\left(\|(-i  \nabla + \mathbf{A})^2 u\|^2 
+  \|V u \|^2 \right)  - (\hat C_{\eps} + 3 C \eps) \|u\|^2\,.
\end{equation}
The claim follows by the density of $\mathcal D$ in $\mathsf{Dom}({\mathscr L})$, see Lemma~\ref{lem.density}.
\end{proof}

\subsection{On Assumption~\ref{a1}}\label{Sec.Ass}
We conclude this section by establishing 
the sufficient condition of Proposition~\ref{prop.a0}.
Note that Theorem~\ref{prop.def-op} is proved under Assumption~\ref{a1},
while our proof of Theorem~\ref{theo.sep} requires the stronger
hypotheses~\eqref{hyp-nabla} and~\eqref{eq.Im-dom}.

\begin{proof}[Proof of Proposition~\ref{prop.a0}]
The proof follows from the fact that, by \eqref{hyp-nabla},
\[
  |\nabla\Phi(x)|^2+|\nabla\Psi(x)|^2\underset{|x|\to+\infty}=o(\mBV(x))
  \,.
\]
Indeed, using in addition~\eqref{eq.Im-dom}, we may write
\begin{eqnarray*}
\lefteqn{\frac {V_{2}^2 +\frac{1}{12d} \abs{\mathbf{B}}^2}{\mBV} 
+ V_{1} -  9 \left( \abs {\nabla \Phi}^2 + \abs{\nabla \Psi}^2 \right)}
\\
&&\geq \frac{1}{12d}\frac{|V|^2+|\mathbf{B}|^2}{\mBV}
+V_{1}-\frac{V^2_{1}}{\mBV} 
- 9 \left( \abs {\nabla \Phi}^2 + \abs{\nabla \Psi}^2 \right)
\\
&&\geq\frac{1}{12d}\mBV-\frac{1}{12d}+o(\mBV)
\,,
\end{eqnarray*}
which provides~\eqref{eq.a0}.
\end{proof}
%

\section{Discrete spectrum and exponential estimates of eigenfunctions}\label{sec.3}
%
The main objective of this section is to establish 
Proposition~\ref{prop.comp-res}
and Theorems~\ref{prop.discrete} and~\ref{theo.main}.

\subsection{Confining potentials}
In addition to Assumption~\ref{a1},
let us assume that~$V$ is confining in the sense of~\eqref{confine}.
\begin{proof}[Proof of Proposition~\ref{prop.comp-res}]
By Theorem~\ref{prop.def-op}, we already know that the resolvent 
of~$\mathscr{L}$ exists at a point of the complex plane.
Hence, it is enough to show that $\Dom(\mathscr{L})$
is compactly embedded in $L^2(\O)$.
Consider~\eqref{eq.coer-Qmu'} with $\mu=0$.  
By the definition of~$\mathscr{L}$ given in~\eqref{operator}
and the Cauchy-Schwarz inequality, 
we get
\[
  \forall u\in\mathsf{Dom}(\mathscr{L})
  \,,\qquad\int_{\Omega} (\gamma_{1}|V|-\gamma_{2})|u|^2\, \der x
  \leq 2\|\mathscr{L} u\|\|u\|\leq \|\mathscr{L} u\|^2+\|u\|^2
  =\|u\|^2_{\mathscr{L}}
\]
for all $u\in\mathsf{Dom}(\mathscr{L})$.
Moreover, we have $\mathsf{Dom}(\mathscr{L})\subset H^2_{\mathsf{loc}}(\Omega)$. 
Thus, by the Riesz-Fr\'echet-Kolmogorov criterion, 
the unit ball for the graph norm of~$\mathscr{L}$ is precompact in $L^2(\Omega)$ 
and thus~$\mathscr{L}$ is an operator with compact resolvent.
\end{proof}

\subsection{General potentials}
Now let us assume only Assumption~\ref{a1}.

\begin{proof}[Proof of Theorem \ref{prop.discrete}]
The inclusion~\eqref{dis1} is again a consequence of~\eqref{eq.coer-Qmu'} 
and Theorem~\ref{th-lax-milgram}. 
It is sufficient to prove~\eqref{dis2}. 
Let
$
  \mu\in\rho_{\gamma_{2}-\gamma_{1}\check V_{\infty}}
$. 
Of course, if $\mu\in\rho_{\gamma_{2}}$ there is nothing to prove. 
Let us define $R>0$ such that, 
\begin{equation}\label{eq.R}
  \forall |x| \geq R 
  \,, \qquad
  |V(x)|\geq \check V_{\infty}
  \,. 
\end{equation}
Then, we have 
\begin{equation} \label{eq-R-g}
  \gamma_{1}|V(x)|-\Re\mu-|\Im\mu|-\gamma_{2}\geq\gamma_{1}\check V_{\infty}
  -\Re\mu-|\Im\mu|-\gamma_{2}=:\gamma>0
  \,.
\end{equation}
for all $|x| \geq R$.
Let us introduce a real-valued smooth function 
with compact support $0\leq \chi\leq 1$ 
such that $\chi(x)=0$ for all $|x| \geq 2R$ 
and $\chi(x)=1$ for all $|x| \leq R$.  
We define
\[
  M:=\left|\gamma_{2}+\Re\mu+|\Im\mu|\right|+1\in[1,+\infty)\,.
\]
Let us write
\[\mathscr{L}-\mu=\mathscr{L}+M\chi-\mu-M\chi\,.\]
We introduce the (closed) operator $\widetilde{\mathscr{L}}:=\mathscr{L}+M\chi$ and, for $\mu \in \C$, the corresponding shifted form $\widetilde Q_\mu := Q_\mu + M \chi$.

Let us explain why $\widetilde{\mathscr{L}}-\mu$ is invertible. 
For that purpose, we recall that by Theorem~\ref{theo-minoration} (with $W=0$) 
and Assumption~\ref{a1}, we have, for all $u\in\mathscr{V}$,
\begin{multline*}
\Re \big[ \widetilde{Q}_{\mu} (u, u) \big] 
+ \Im \big[ \widetilde{Q}_{\mu} (u, \Phi u) \big]
\\
\geq \frac 12 \nr{(-i\nabla +  \mathbf{A}) u}^2
+\int_{\Omega}\big(M\chi+\gamma_{1}|V|-\gamma_{2}-\Re\mu-|\Im\mu|\big)|u|^2
\,\der x
\,.
\end{multline*}
By the definitions of~$M$ and~$\gamma$, we deduce that 
\[
  \Re \big[ \widetilde{Q}_{\mu} (u, u) \big] 
  + \Im \big[ \widetilde{Q}_{\mu} (u, \Phi u) \big]
  \geq \frac 12 \nr{(-i\nabla +  \mathbf{A}) u}^2
  +\min(1,\gamma)\nr u^2\,.
\]
This proves the coercivity of $\widetilde{Q}_{\mu}$ on $\mathscr{V}$ and thus, by Theorem \ref{th-lax-milgram}, $\widetilde{\mathscr{L}}-\mu$ is invertible.

Now, the multiplication operator~$M\chi$ is a relatively compact perturbation 
of~$\widetilde{\mathscr{L}}-\mu$. Therefore, by Lemma~\ref{lem.rel.comp}, 
$\mathscr{L}-\mu$ is a Fredholm operator with index~$0$. 
From Lemma~\ref{lem.gru}, we deduce that the spectrum in 
$\rho_{\gamma_{2}-\gamma_{1}\check V_{\infty}}$ is discrete 
(that is, made of isolated eigenvalues
of finite algebraic multiplicity, 
see Appendix~\ref{sec.A}).
\end{proof}

\subsection{Agmon-type estimates}
Theorem~\ref{theo.main} is essentially a consequence of the following proposition
about properties of solutions of an inhomogeneous equation in a weighted space.
\begin{proposition}\label{prop.inhomogen}
Let 
$
  \lambda \in 
  \mathsf{sp}(\mathscr{L})\cap\rho_{\gamma_{2}-\gamma_{1}\check V_{\infty}}\neq\emptyset
$. 
Let us consider $\psi_0 \in L^2(\O)$ such that 
\begin{equation} \label{eq-agmon-psi0}
e^{\frac {1-\eps} 3 \, \der_{\mathsf{Ag}}(x)} \psi_0 \in L^2(\O)
\end{equation}
for some $\eps \in (0,1)$
and assume that $\psi \in \mathsf{Dom}(\mathscr{L})$ satisfies 
\begin{equation}\label{eq.inhomogen}
\mathscr{L} \psi = \lambda\psi + \psi_0\,.
\end{equation}
Then
\begin{equation} \label{eq-agmon-psi}
e^{\frac {1-\eps} 3 \, \der_{\mathsf{Ag}}(x)} \psi \in L^2(\O)\,.
\end{equation}
\end{proposition}
\begin{proof}
By Theorem \ref{prop.discrete}, 
$\lambda$~is an eigenvalue of finite algebraic multiplicity. 
Given $W \in W^{1,\infty}(\O;\R)$, we have
\begin{align*}
  \Re Q(\psi,e^{2W}\psi)
  &=\Re(\lambda) \|e^{W}\psi\|^2 
  + \Re \innp{e^W \psi_0}{e^W \psi} 
  \,,
  \\
  \Im Q(\psi,\Phi e^{2W}\psi)
  &=\Im(\lambda) \int_{\Omega}\Phi e^{2W}|\psi|^2\,\der x 
  + \Im \innp{e^W \psi_0}{\Phi e^W \psi}
  \,.
\end{align*}
By Theorem~\ref{theo-minoration} (with $\mu=0$) and Assumption~\ref{a1},
\[
  \big( \Re(\lambda) + \abs{\Im(\lambda)} \big) \nr{e^W \psi}^2 
  \geq \int_\O  \left(\gamma_{1}|V|-\gamma_{2}-9 \abs{\nabla W}^2 \right)
  \abs{e^W \psi}^2 \, \der x - \nr{e^W \psi_0}\nr{e^W \psi}
  \,.
\]
Thus, we get
\[
\int_\O  \left(\gamma_{1}|V|- \Re(\lambda) 
- \abs{\Im(\lambda)}-\gamma_{2}-9 \abs{\nabla W}^2 \right)\abs{e^W \psi}^2 \, \der x
\leq  \nr{e^W \psi_0}\nr{e^W \psi}\,.
\]
Let $R$ be as in \eqref{eq.R}. Splitting the integral into two parts, we get
\begin{multline*}
\int_{\{|x|>R\}}  
\left(
\gamma_{1}|V|- \Re(\lambda) - \abs{\Im(\lambda)}-\gamma_{2}-9 \abs{\nabla W}^2 
\right)\abs{e^W \psi}^2 \, \der x
\\
\leq \int_{\{|x|<R\}}  
\left(
-\gamma_{1}|V|+ \Re(\lambda) + \abs{\Im(\lambda)}+\gamma_{2}+9 \abs{\nabla W}^2
 \right)\abs{e^W \psi}^2 \, \der x + \nr{e^W \psi_0}\nr{e^W \psi}
\,,
\end{multline*}
so that, for some $C > 0$, we have by~\eqref{eq.Agmon},
\begin{multline}\label{eq.split}
\int_{\{|x|>R\}}  
\left( 
\abs{\nabla \der_{\mathsf{Ag}}(x)}^2 - 9 \abs{\nabla W}^2\right) \abs{e^W \psi}^2 
\, \der x
\\
\leq  \int_{\{|x|<R\}}  
\left(C + 9 \abs{\nabla W}^2 \right)\abs{e^W \psi}^2 \, \der x
+ \nr{e^W \psi_0}\nr{e^W \psi}
\,.
\end{multline}
We set $\eta := \frac {\sqrt{1-\eps}} 3$ and we consider the functions $(\chi_{n})_{n\geq 1}$ defined as follows
\[\chi_{n}(s):=\begin{cases}
s& \mbox{for } 0\leq s\leq n\,,\\
2n-s& \mbox{for } n\leq s\leq 2n\,,\\
0& \mbox{for } s\geq 2n\,.
\end{cases}\]
Note that $|\chi'_{n}(s)|=1$ a.e. on $[0, 2n]$ and $|\chi'_{n}(s)|=0$ for $s> 2n$.

Then for $n\geq 1$ and $x \in \O$ we set
\[
  W_{n}(x) := \eta \, \chi_{n} (\der_{\mathsf{Ag}}(x))\,.
\]
We have
\[
\nabla W_{n}(x)=\eta \, \chi'_{n} (\der_{\mathsf{Ag}}(x))\nabla \der_{\mathsf{Ag}}(x)
\]
and 
\[
|\nabla W_{n}(x)|^2
\leq\eta^2 |\nabla \der_{\mathsf{Ag}}(x)|^2 = \frac {1-\eps} 9 \abs{\nabla \der_{\mathsf{Ag}}(x)}^2 \,.
\]
By \eqref{eq.split} we obtain that there exists $C > 0$ such that, for all $n\geq 1$,
\[
\int_{\{|x| > R\}}  
\eps \abs{\nabla \der_{\mathsf{Ag}}(x)}^2 \abs{e^{W_{n}} \psi}^2 
\, \der x\leq C\|\psi\|^2 + \nr{e^{W_n}\psi} \nr{e^{W_n}\psi_0}
\,,
\]
and therefore, by \eqref{eq-R-g},
\[
\int_{\{|x| > R\}}  
\eps\gamma \abs{e^{W_{n}} \psi}^2 \, \der x 
\leq C\|\psi\|^2 
+ \frac {\eps \gamma} 2 \nr{e^{W_n}\psi}^2 
+ \frac 1 {2 \eps \gamma} \nr{e^{W_n}\psi_0}^2\,.
\]
For another constant $C>0$ independent of~$n$, 
we get
\[
  \int_{\Omega} \abs{e^{W_{n}} \psi}^2 \, \der x
  \leq C\|\psi\|^2 + C\nr{e^{W_n}\psi_0}^2
  \,.
\]
It remains to take the limit $n\to+\infty$ 
and use the Fatou lemma to conclude.
\end{proof}

Now we are in a position to prove the main result of this paper.
\begin{proof}[Proof of Theorem~\ref{theo.main}]
If $\psi \in \Ker(\mathscr{L}-\lambda)$, we apply Proposition \ref{prop.inhomogen} with $\psi_0 = 0$ to deduce that $\psi$ satisfies \eqref{eq.est-Agmon}. 

Let us now explain why this conclusion holds also for the algebraic eigenspace 
(see Appendix~\ref{sec.A}).
Let us consider $\psi$ in this space.

We have
\[
  (\mathscr{L}-\lambda)^r\psi=0
  \qquad\mbox{with}\qquad
  r:=\dim\Ran(P_{\lambda})\geq 1
  \,.
\]
Now, we proceed by induction. Consider $k \in \Ii 1 r$ and assume that 
\[
  (\mathscr{L}-\lambda)^k\psi
  \in L^2\left(\Omega, e^{\frac {1-\eps} 3 \der_{\mathsf{Ag}}(x)}\,\der x\right)
  \,.
\]
Then, we write
\[
  (\mathscr{L}-\lambda)\left\{(\mathscr{L}-\lambda)^{k-1}\psi\right\}
  =(\mathscr{L}-\lambda)^k\psi
  \,.
\]
We are in the situation \eqref{eq.inhomogen} and we deduce that
\[
  (\mathscr{L}-\lambda)^{k-1}\psi
  \in L^2\left(\Omega, e^{\frac {1-\eps} 3 d_{\mathsf{Ag}}(x)}\,\der x\right)
  \,.
\]
This concludes the proof.
\end{proof}

\appendix
\section{Reminders of spectral theory}\label{sec.A}
%
Since spectral theory of non-self-adjoint operators 
is less unified than its self-adjoint sister,
in this appendix we collect some notions 
used throughout the paper.
We refer to standard monographs \cite{Kato}, \cite[Chap.~I.3,~IX]{Edmunds-Evans} and \cite[Chap.~XVII]{Gohberg-1990}
or a recent summary~\cite{KS-book}
for a more comprehensive exposition.

Let~$\mathsf{H}$ be a Hilbert space.
An operator $\mathscr{M}:\mathsf{Dom}(\mathscr{M})\to\mathsf{H}$ 
is said to be \emph{Fredholm} when
$\Ker(\mathscr{M})$ finite-dimensional
and $\Ran(\mathscr{M})$ is closed with finite codimension.
Then the \emph{index} of~$\mathscr{M}$ is defined by
$
  \ind(\mathscr{M})
  :=\dim\Ker(\mathscr{M}) - \codim\Ran(\mathscr{M})
$.
When $\mathsf{Dom}(\mathscr{M})$ is dense in~$\mathsf{H}$, 
we may classically define the adjoint $\mathscr{M}^*$ of $\mathscr{M}$ 
and then we have 
$
  \dim\Ker(\mathscr{M}^*) = \codim\Ran(\mathscr{M})
$.
We denote by $\mathsf{Fred}_{0}(\mathscr{M})$ the set of 
all complex numbers~$\lambda$ such that $\mathscr{M}-\lambda$ 
is a Fredholm operator with index~$0$.
 
Let~$\mathscr{M}$ be an arbitrary closed operator in~$\mathsf{H}$.
The \emph{spectrum} $\mathsf{sp}(\mathscr{M})$
is defined as the set of all complex numbers~$\lambda$ 
such that $\mathscr{M}-\lambda$ is not bijective
as an operator from $\Dom(\mathscr{M})$ to~$\mathsf{H}$. 
The \emph{resolvent set} $\rho(\mathscr{M})$ is the complement 
of the spectrum in the complex plane.
We call the intersection
$
  \mathsf{sp}_{\mathsf{fre}}(\mathscr{M})
  := \mathsf{sp}(\mathscr{M}) \cap \mathsf{Fred}_{0}(\mathscr{M})
$
the \emph{Fredholm spectrum} and define the \emph{essential spectrum}
by the complement  
$
  \mathsf{sp}_{\mathsf{ess}}(\mathscr{M})
  := \mathsf{sp}(\mathscr{M})\setminus\mathsf{sp}_{\mathsf{fre}}(\mathscr{M})
$
(it is the essential spectrum due to Schechter
denoted by $\mathsf{sp}_{\mathsf{e}4}(\mathscr{M})$ in~\cite{Edmunds-Evans}).
Finally, we define the \emph{discrete spectrum} $\mathsf{sp}_{\mathsf{dis}}(\mathscr{M})$
to be the set of all isolated eigenvalues~$\lambda$ 
for which the \emph{algebraic} (or \emph{root}) eigenspace 
$
  \cup_{k=1}^\infty \Ker([\mathscr{M}-\lambda]^k)
$
is finite-dimensional and such that 
$\mathscr{M}-\lambda$ has a closed range.
The elements of $\mathsf{sp}_{\mathsf{dis}}(\mathscr{M})$
are called the \emph{discrete eigenvalues} of~$\mathscr{M}$.

Let~$\lambda$ be an isolated eigenvalue of~$\mathscr{M}$.
Another characterisation of~$\lambda$ to belong to the discrete spectrum 
is through the \emph{eigenprojection} 
\begin{equation}\label{projection}
  P_{\lambda}:=\frac{1}{2i\pi}\int_{\Gamma_{\lambda}}(z-\mathscr{M})^{-1}\,\der z
  \,,
\end{equation}
where $\Gamma_{\lambda}$ is a contour that enlaces only~$\lambda$
as an element of the spectrum.
$P_{\lambda} : \mathsf{H}\to\mathsf{Dom}(\mathscr{M})\subset\mathsf{H}$ 
is a bounded operator which commutes with~$\mathscr{M}$ 
and does not depend on~$\Gamma_{\lambda}$.
We say that~$\lambda$ has \emph{finite algebraic multiplicity} 
when the range of~$P_{\lambda}$ is finite-dimensional. 
In this case, $\lambda$~is a discrete eigenvalue of~$\mathscr{M}$.
Moreover, the range of~$P_{\lambda}$ coincides with 
the algebraic eigenspace of~$\lambda$.
It is an invariant subspace of~$\mathscr{M}$ of finite dimension 
and such that the spectrum of $\mathscr{M}_{|\Ran(P_{\lambda})}$ equals $\{\lambda\}$.

Finally, we recall three standard results. For the proofs see \cite[Chap.~I.3]{Edmunds-Evans}, \cite[Thm.~IX.2.1]{Edmunds-Evans} and \cite[Thm.~XVII.2.1]{Gohberg-1990}, respectively.

\begin{lemma}
Let $(\mathscr{M},\mathsf{Dom}(\mathscr{M}))$ be a closed operator 
in a Hilbert space~$\mathsf{H}$.
Let us equip $\mathsf{Dom}(\mathscr{M})$ with the graph norm $\|\cdot\|_{\mathscr{M}}$,
which makes $(\mathsf{Dom}(\mathscr{M}),\|\cdot\|_{\mathscr{M}})$ a new Hilbert space.
Let~$\mathcal{M}$ be the operator~$\mathscr{M}$ reconsidered
as an operator from $(\mathsf{Dom}(\mathscr{M}),\|\cdot\|_{\mathscr{M}})$ to~$\mathsf{H}$.
The following properties hold:
\begin{enumerate} 
\item[\emph{(i)}] $\mathcal{M}$ is bounded,
\item[\emph{(ii)}] $\mathcal{M}$ is Fredholm 
if and only if~$\mathscr{M}$ is Fredholm.
In this case, $\ind(\mathcal{M})=\ind(\mathscr{M})$.
\end{enumerate}
\end{lemma}
\begin{lemma}\label{lem.rel.comp}
Let $(\mathscr{M},\mathsf{Dom}(\mathscr{M}))$ be a closed invertible operator
and $(\mathscr{P},\mathsf{Dom}(\mathscr{M}))$ another operator 
in a common Hilbert space~$\mathsf{H}$. 
Assume that
$(\mathscr{M}+\mathscr{P},\mathsf{\mathsf{Dom}}(\mathscr{M}))$ is closed 
and $\mathscr{P}\mathscr{M}^{-1}$ is compact. 
Then the operator $(\mathscr{M}+\mathscr{P},\mathsf{\mathsf{Dom}}(\mathscr{M}))$ 
is Fredholm and $\ind(\mathscr{M}+\mathscr{P})=\ind(\mathscr{M})=0$.
\end{lemma}

\begin{lemma}\label{lem.gru}
Let $(\mathscr{M},\mathsf{Dom}(\mathscr{M}))$ be a closed operator 
in a Hilbert space~$\mathsf{H}$ with a non-empty resolvent set 
and let $\triangle$ be an open connected subset of 
\begin{equation*}
\{ z \in \C \; : \, \mathscr M - z \ \text{ is Fredholm} \}.
\end{equation*}
If $\triangle \cap \rho(\mathscr M) \neq \emptyset$, 
then  $\mathsf{sp}(\mathscr M) \cap\triangle$ 
is a countable set, 
with no accumulation point in $\triangle$, 
consisting of eigenvalues of~$\mathscr M$ 
with finite algebraic multiplicities.
\end{lemma}

\newcommand{\etalchar}[1]{$^{#1}$}


\begin{thebibliography}{BBDCZ10}

\bibitem[AAD01]{Abramov-Aslanyan-Davies_2001}
A.~A. Abramov, A.~Aslanyan, and E.~B. Davies.
\newblock Bounds on complex eigenvalues and resonances.
\newblock {\em J. Phys. A: Math. Gen.}, 34:57--72, 2001.

\bibitem[Agm85]{Agmon85}
S.~Agmon.
\newblock Bounds on exponential decay of eigenfunctions of {S}chr\"odinger
  operators.
\newblock In {\em Schr\"odinger operators ({C}omo, 1984)}, volume 1159 of {\em
  Lecture Notes in Math.}, pages 1--38. Springer, Berlin, 1985.

\bibitem[AH14]{AH2014}
Y.~Almog and B.~Helffer.
\newblock Global stability of the normal state of superconductors in the
  presence of a strong electric current.
\newblock {\em Comm. Math. Phys.}, 330:1021--1094, 2014.

\bibitem[AH15]{AlmogHe15}
Y.~{Almog} and B.~{Helffer}.
\newblock {On the spectrum of non-selfadjoint Schr\"odinger operators with
  compact resolvent.}
\newblock {\em {Commun. Partial Differ. Equations}}, 40(8):1441--1466, 2015.

\bibitem[BBDCZ10]{Bonnet-BenDhia-2010-234}
A.~S. Bonnet-Ben~Dhia, P.~Ciarlet, Jr., and C.~M. Zw{\"o}lf.
\newblock Time harmonic wave diffraction problems in materials with
  sign-shifting coefficients.
\newblock {\em J. Comput. Appl. Math.}, 234(6):1912--1919, 2010.

\bibitem[BH99]{Brown-1999-124}
R.~C. Brown and D.~B. Hinton.
\newblock Two separation criteria for second order ordinary or partial
  differential operators.
\newblock {\em Math. Bohem.}, 124(2-3):273--292, 1999.

\bibitem[BR14]{boucletr14}
J.-M. Bouclet and J.~Royer.
\newblock Local energy decay for the damped wave equation.
\newblock {\em J. Funct. Anal.}, 266(2):4538--4615, 2014.

\bibitem[BST15]{Boegli-2015}
S.~B{\"o}gli, P~Siegl, and C.~Tretter.
\newblock {Approximations of spectra of Schr\"odinger operators with complex
  potential on $\R^d$}.
\newblock arXiv: 1512.01826, 2015.

\bibitem[CGR10]{Cappiello-2010-111}
M.~Cappiello, T.~Gramchev, and L.~Rodino.
\newblock {Entire extensions and exponential decay for semilinear elliptic
  equations}.
\newblock {\em J. Anal. Math.}, 111:339--367, 2010.

\bibitem[Dav95]{Davies-1995}
E.~B. Davies.
\newblock {\em Spectral theory and differential operators}, volume~42 of {\em
  Cambridge Studies in Advanced Mathematics}.
\newblock Cambridge University Press, Cambridge, 1995.

\bibitem[DS88]{DS2}
N.~Dunford and J.~T. Schwartz.
\newblock {\em {L}inear {O}perators, {S}pectral {T}heory, {S}elf {A}djoint
  {O}perators in {H}ilbert {S}pace, {P}art 2}.
\newblock John Wiley \& Sons, Inc., New York, 1988.

\bibitem[Duf83]{Duf83}
A.~Dufresnoy.
\newblock Un exemple de champ magn\'etique dans {${\bf R}^{\nu }$}.
\newblock {\em Duke Math. J.}, 50(3):729--734, 1983.

\bibitem[EE87]{Edmunds-Evans}
D.~E. Edmunds and W.~D. Evans.
\newblock {\em Spectral theory and differential operators}.
\newblock Oxford University Press, Oxford, 1987.

\bibitem[EG78]{Everitt-1978-79}
W.~N. Everitt and M.~Giertz.
\newblock Inequalities and separation for {S}chr\"odinger type operators in
  {$L_{2}({\bf R}^{n})$}.
\newblock {\em Proc. Roy. Soc. Edinburgh Sect. A}, 79(3-4):257--265, 1977/78.

\bibitem[Exn85]{Exner-open}
P.~Exner.
\newblock {\em Open Quantum Systems and {F}eynman Integrals}.
\newblock D. Reidel Publishing Company, Dordrecht, 1985.

\bibitem[EZ78]{Evans-1978-8}
W.~D. Evans and A.~Zettl.
\newblock Dirichlet and separation results for {S}chr\"odinger-type operators.
\newblock {\em Proc. Roy. Soc. Edinburgh Sect. A}, 80(1-2):151--162, 1978.

\bibitem[FKV16]{FKV}
L.~Fanelli, D.~Krej\v{c}i\v{r}\'{\i}k, and L.~Vega.
\newblock Spectral stability of {S}chr{\"o}dinger operators with subordinated
  complex potentials.
\newblock {\em J. Spectr. Theory}, 2016.

\bibitem[GGK90]{Gohberg-1990}
I.~Gohberg, S.~Goldberg, and M.~A. Kaashoek.
\newblock {\em Classes of Linear Operators}, volume~1.
\newblock Birkh{\"a}user Verlag Basel, 1990.

\bibitem[GKMV13]{Kostrykin_2013}
L.~Grubi\v{s}i\'c, V.~Kostrykin, K.~A. Makarov, and K.~Veseli\'c.
\newblock Representation theorems for indefinite quadratic forms revisited.
\newblock {\em Mathematika}, 59:169--189, 2013.

\bibitem[HM88]{HM88}
B.~Helffer and A.~Mohamed.
\newblock Caract\'erisation du spectre essentiel de l'op\'erateur de
  {S}chr\"odinger avec un champ magn\'etique.
\newblock {\em Ann. Inst. Fourier (Grenoble)}, 38(2):95--112, 1988.

\bibitem[Kat66]{Kato}
T.~Kato.
\newblock {\em Perturbation Theory for Linear Operators}.
\newblock Springer-Verlag, Berlin, 1966.

\bibitem[KS15]{KS-book}
D.~Krej\v{c}i\v{r}\'{\i}k and P.~Siegl.
\newblock {\em Elements of spectral theory without the spectral theorem}.
\newblock In \emph{Non-selfadjoint operators in quantum physics: Mathematical
  aspects} (432 pages), F. Bagarello, J.-P. Gazeau, F. H. Szafraniec, and M.
  Znojil, Eds. Wiley-Interscience, 2015.

\bibitem[Per60]{Persson60}
A.~Persson.
\newblock Bounds for the discrete part of the spectrum of a semi-bounded
  {S}chr\"odinger operator.
\newblock {\em Math. Scand.}, 8:143--153, 1960.

\bibitem[Ray16]{Raymond}
N.~Raymond.
\newblock {\em Bound states of the {M}agnetic {S}chr{\"o}dinger {O}perator}.
\newblock EMS Tracts in Mathematics. 2016.

\bibitem[RBM{\etalchar{+}}12]{Regensburger_2012}
A.~Regensburger, Ch. Bersch, M.-A. Miri, G.~Onishchukov, D.~N. Christodoulides,
  and U.~Peschel.
\newblock Parity-time synthetic photonic lattices.
\newblock {\em Nature}, 488:167--171, 2012.

\bibitem[Sib75]{Sibuya-1975}
Y.~Sibuya.
\newblock {\em {Global theory of a second order linear ordinary differential
  equation with a polynomial coefficient}}.
\newblock North-Holland Publishing Co., Amsterdam, 1975.

\bibitem[Sj{\"o}00]{sjostrand00}
J.~Sj{\"o}strand.
\newblock Asymptotic distribution of eigenfrequencies for damped wave
  equations.
\newblock {\em Publ. RIMS, Kyoto Univ.}, 36:573--611, 2000.

\bibitem[SK12]{SK}
P.~Siegl and D.~Krej\v{c}i\v{r}\'{i}k.
\newblock On the metric operator for the imaginary cubic oscillator.
\newblock {\em Phys. Rev. D}, 86:121702(R), 2012.

\bibitem[tESV15]{Elst-2015-269}
A.~F.~M. ter Elst, M.~Sauter, and H.~Vogt.
\newblock {A generalisation of the form method for accretive forms and
  operators}.
\newblock {\em J. Funct. Anal.}, 269:705--744, 2015.

\end{thebibliography}
\end{document}